\documentclass[3p]{article}
\usepackage{graphicx} 
\usepackage{amssymb}
\usepackage{amsthm}
\usepackage{amsmath}
\usepackage{pgfplots}
\usepackage{CJKutf8}
\usepackage{setspace}
\usepackage{titlesec}
\allowdisplaybreaks[4]
\date{}

\theoremstyle{plain}
\newtheorem{thm}{Theorem}[section]
\newtheorem{lem}[thm]{Lemma}
\newtheorem{cor}[thm]{Corollary}
\newtheorem{prop}[thm]{Proposition}

\newtheorem*{rem}{Remark}
\newtheorem*{thmA}{Theorem A}
\newtheorem*{thmB}{Theorem B}

\newtheorem{defi}[thm]{Definition}

\begin{document}
\title{Weighted inequalities for variation operators associated with fractional Schrödinger semigroups}
\author{Yanhan Chen \footnote{Department of Mathematics, Graduate School of Science, Kyoto University, Kyoto 606-8502, Japan. \newline e-mail: \texttt{chen.yanhan.67s@st.kyoto-u.ac.jp}}}
\maketitle
\renewcommand{\abstractname}{Abstract}
\begin{abstract}
    Let $\{e^{-tL^{\alpha}}\}_{t>0}$ be the fractional Schr\"{o}dinger semigroup associated with $L=-\Delta+V$, where $V$ is a non-negatvie potential belonging to the reverse H\"{o}lder class. In this paper, we establish weighted boundedness properties of the variation operator related to $\{e^{-tL^{\alpha}}\}_{t>0}$, including weighted $L^{p}-L^{q}$ quantitative inequalities and mixed weak-type inequalities.
\end{abstract}


\section{Introduction}
\quad Let us consider the Schr\"{o}dinger operator $L=-\Delta+V$ in $\mathbb{R}^{n}$ with $n\geq3$, where $\Delta$ denotes the Laplace operator $\Delta=\sum_{i=1}^{n}\partial^{2}_{i}$, and the potential $V$ is non-negative and belongs to the reverse H\"{o}lder class $RH_{s}$ for some $s>n/2$. That is, there exists $C>0$ such that
$$\left(\frac{1}{|B|}\int_{B}V(x)^{s}dx\right)^{\frac{1}{s}}\leq\frac{C}{|B|}\int_{B}V(x)dx,$$
for every ball $B$ in $\mathbb{R}^{n}$. It is well known that the Schr\"{o}dinger operator $L$ generates a Schr\"{o}dinger semigroup $\{e^{-tL}\}_{t>0}$ on $L^{2}$. We mainly focus on the fractional Schr\"{o}dinger semigroup $\{e^{-tL^{\alpha}}\}_{t>0}$ with $0<\alpha<1$, which is defined as the $C_{0}$-semigroup on $L^{2}$ whose generator is $-L^{\alpha}$.

In this paper we study the weighted inequalities for the variation operators associated with the fractional Schr\"{o}dinger semigroups. Let $\left\{T_{t}\right\}_{t>0}$ be a family of operators defined in some function spaces such that $\lim_{t\rightarrow0}T_{t}$ exists in some sense. For $a>2$, we denote the variation operator associated with $\{T_{t}\}_{t>0}$ as $V_{a}(T_{t})$, which is defined by
\begin{equation*}
    V_{a}(T_{t})(f)(x):=\sup_{\{t_{j}\}\searrow0}\left(\sum_{i=1}^{\infty}\left|T_{t_{i}}f(x)-T_{t_{i+1}}f(x)\right|^{a}\right)^{\frac{1}{a}}.\tag{1.1}
\end{equation*}
The investigation of variation operators originated in the work of Lépingle \cite{L76} and was further developed by Bourgain \cite{B89}, who derived variation estimates for ergodic averages in dynamical systems. Since then, 
inequalities for variation operators have been extensively studied in fields such as probability, harmonic analysis, and ergodic theory. A key motivation for examining such inequalities for a family of operators is that they provide insights into both the speed and the manner of convergence of the family of operators under consideration.

In recent years, extensive research has been conducted on inequalities for variation operators within the fields of harmonic analysis and ergodic theory. Campbell, Jones, Reinhold and Wierdl \cite{CJRW00} studied the oscillation and variation operator norms for certain classes of convolution operators, they established the $L^{p}$ inequality of variation operators associated with Hilbert transform. For commutators, Zhang and Wang \cite{ZW15} investigated the boundedness of the oscillation and variation operators for the commutators generated by Calderón-Zygmund singular integrals with Lipschitz functions in the weighted $L^{p}$ spaces. Subsequently, in \cite{ZW16}, the same authors also examined the boundedness properties of the oscillation and variation operators for Calderón-Zygmund singular integrals and the corresponding commutators on the weighted Morrey spaces. For further results on singular integrals, Chen, Ding, Hong and Liu \cite{CDHL18} obtained the variational inequalities for singular integrals and averaing operators with rough kernels. For the Schr\"{o}dinger settings, Betancor, Fariña, Harboure and Rodríguez-Mesa \cite{BFHR13} established $L^{p}$ bounds for the variation operator associated with Schr\"{o}dinger semigroups that
$$\|V_{a}(e^{-tL})f\|_{L^{p}\rightarrow L^{P}}<\infty\quad (p>1);\quad\|V_{a}(e^{-tL})f\|_{L^{1}\rightarrow L^{1,\infty}}<\infty.$$
Subsequently, Tang and Zhang \cite{TZ16} extended their result to the weighted setting, and Wen and Wu \cite{WW25} further explored the two-weighted cases. Most recently, Wang, Zhao, Li and Liu \cite{WZLL25} studied the weighted $L^{p}$ boundedness of the variation operator associated with fracitonal Schr\"{o}dinger semigroups. Specifically, they established that
\begin{thmA}
    \textnormal{(Wang, Zhao, Li and Liu \cite{WZLL25})} Let $a>2$, $0<\alpha<1$. Assume that $1<p<\infty$, $w\in A_{p}^{\rho}$. Then the variation operator $V_{a}(e^{-tL^{\alpha}})$ is bounded from $L^{p}(w)$ into itself.
\end{thmA}
\begin{thmB}
    \textnormal{(Wang, Zhao, Li and Liu \cite{WZLL25})} Let $a>2$, $0<\alpha<1$. Then the variation operator $V_{a}(e^{-tL^{\alpha}})$ is bounded from $L^{1}$ to $L^{1,\infty}$.
\end{thmB}
Our first result is a generalization of Theorem A. We establish quantitative weighted $L^{p}$ boundedeness for $V_{a}(e^{-tL^{\alpha}})$. More precisely, we prove the following:
\begin{thm}
    Let $n\geq3$, $V\in RH_{s}$ with $s>n/2$ and $L=-\Delta+V$. Let $a>2$, $1<p<\infty$, $0<\alpha<1$ and $\rho$ be the critical radius function associated with $V$, then\\
\textnormal{(i)} For $w\in A_{p}^{\rho,\theta}$, we have 
    $$\|V_{a}(e^{-tL^{\alpha}})f\|_{L^{p}(w)}\lesssim[w]_{A_{p}^{\rho,\theta}}^{\textnormal{max}\left\{1,\frac{1}{p-1}\right\}}\|f\|_{L^{p}(w)}.$$
\textnormal{(ii)} Let $\Psi$ be a Young function such that $\bar{\Psi}\in B_{p}$, and a pair $(u,v)$ of weights satisfies
    $$[u,v]_{\Psi,p,\rho,\theta}:=\sup_{Q}\|v^{\frac{1}{p}}\|_{p,Q}\|u^{-\frac{1}{p}}\|_{\Psi,Q}\psi_{\theta}(Q)^{-1}<\infty.$$
We have
$$\|V_{a}(e^{-tL^{\alpha}})f\|_{L^{p}(v)}\lesssim[u,v]_{\Psi,p.\rho,\theta}[\bar{\Psi}]_{B_{p}}^{\frac{1}{p}}\|f\|_{L^{p}(u)}.$$
\end{thm}
Regarding the weak type estimate for $V_{a}(e^{-tL^{\alpha}})$, we turn to the study of a more genernal mixed weak type inequality. In 1985, E. Sawyer \cite{S85} proved that if both $u$ and $v$ are $A_{1}$ weights, then the inequality
\begin{equation}
uv\left(\left\{x\in\mathbb{R}:\frac{M(fv)(x)}{v(x)}>t\right\}\right)\lesssim\frac{1}{t}\int_{\mathbb{R}}|f(y)|duv(y)\tag{1.2}
\end{equation}
holds for every $t>0$, where $M$ denotes the classical Hardy-Littlewood maximal operator. This estimate is highly non-trivial extension of the classical weak type $(1,1)$ inequality for $M$ due to the presence of the weight function $v$ within the distribution set. Moreover, several other noteworthy applications exist, including estimates for multilinear operators and commutators with BMO functions. Subsequently, Cruz-Uribe, Martell and P\'{e}rez \cite{CMP05} extended this result to higher dimensions and under weaker condition on the weight. They proved that for $u\in A_{1}$ and $v\in A_{\infty}(u)$, it holds that
\begin{equation}
    uv\left(\left\{x\in\mathbb{R}^{n}:\frac{T(fv)(x)}{v(x)}>t\right\}\right)\lesssim\frac{1}{t}\int_{\mathbb{R}^{n}}|f(y)|duv(y),\tag{1.3}
\end{equation}
where operator $T$ is either the Hardy-Littlewood maximal operator or a Calder\'{o}n-Zygmund singular integral. Most recently, Li, Ombrosi and P\'{e}rez \cite{LOP19} further extended (1.3) to a broader class of weights. They confirmed inequality (1.3) for the maximal operator under the weakened weight condition $u\in A_{1}$ and $v\in A_{\infty}$, there by resolving a long standing conjecture proposed by Sawyer. For classes of operators and weights associated
 with a critical radius function, Berra, Pradolini and Quijano \cite{BPQ25} established mixed weak type inequalities for the maximal operator. Subsequently, Wen and Wu \cite{WW25} studied inequality (1.3) for the operator $V_{a}(e^{-tL})$. They demonstrated that for $u\in A_{1}^{\rho}$ and $v\in A_{\infty}^{\rho}(u)$, the inequality holds that
 \begin{equation}
 uv\left(\left\{x\in\mathbb{R}^{n}:\frac{V_{a}(e^{-tL})(fv)(x)}{v(x)}>t\right\}\right)\lesssim\frac{1}{t}\int_{\mathbb{R}^{n}}|f(y)|duv(y).\tag{1.4}
 \end{equation}
Our second result addresses mixed weak type estimates for the variation operator associated with fractional Schr\"{o}dinger semigroups. Specifically, we prove the following:
\begin{thm}
Let $n\geq3$, $V\in RH_{s}$ with $s>n/2$ and $L=-\Delta+V$. Let $a>2$, $0<\alpha<1$ and $\rho$ be the critical radius function associated with $V$. Then for $u\in A_{1}^{\rho}$ and $v\in A_{\infty}^{\rho}$, there exist $C>0$ such that
$$\left\|\frac{V_{a}(e^{-tL^{\alpha}})(fv)}{v}\right\|_{L^{1,\infty}(uv)}\leq C\|f\|_{L^{1}(uv)}.$$
\end{thm}
\begin{rem}
    The conditions $u\in A_{1}^{\rho}$ and $v\in A_{\infty}^{\rho}$ are weaker than $u\in A_{1}^{\rho}$ and $v\in A_{\infty}^{\rho}(u)$ in \textnormal{\cite{WW25}}. Indeed, the latter implies there exist $C>0$ and $\eta,\delta,\epsilon>0$ such that
    $$\frac{v(E)}{v(Q)}\leq C\left(\frac{u(E)}{u(Q)}\right)^{\epsilon}\psi_{\eta}(Q)\leq C\left(\frac{|E|}{|Q|}\right)^{\epsilon\delta}\psi_{\eta}(Q)$$
for any cube $Q$ and $E\subseteq Q$, which further implies that $v\in A_{\infty}^{\rho}$ (see the detailed definition and properties of weight class $A_{p}^{\rho}$ in Section 2.1).  
\end{rem}
By setting $v=1$, we obtain a weighted weak type $(1,1)$ inequality for $V_{a}(e^{-tL^{\alpha}})$, which serves as a weighted generalization of Theorem B.
\begin{cor}
    Let $n\geq3$, $V\in RH_{s}$ with $s>n/2$ and $L=-\Delta+V$. Let $a>2$, $0<\alpha<1$ and $\rho$ be the critical radius function associated with $V$. Then for $u\in A_{1}^{\rho}$, there exist $C>0$ such that
$$\left\|{V_{a}(e^{-tL^{\alpha}})f}\right\|_{L^{1,\infty}(u)}\leq C\|f\|_{L^{1}(u)}.$$
\end{cor}
\begin{rem}
We prove Theorem 1.2 using an extrapolation method (see details in Section 4). Although a careful examination of the extrapolation theorem would yield an implicit constant expressed in terms of $[u]_{A_{1}^{\rho}}$ and $[v]_{A_{\infty}^{\rho}}$, this constant remains far from optimal. For this reason, we have chosen not to include the implicit constant in the statement of Theorem 1.2. Nevertheless, obtaining an optimal bound for the mixed weak-type inequality is also meaningful, since our weak-type estimate in Corollary 1.3 is derived from the mixed weak-type inequalities in Theorem 1.2. Moreover, the pursuit of optimal constants in weak $(1,1)$ inequalities represents a fundamental problem in harmonic analysis. Therefore, we consider the quantitative mixed weak type inequalities (or particularly, weak type inequalities) for $V_{a}(e^{-tL^{\alpha}})$ to remain a meaningful open question.
\end{rem}
This paper is organized as follows. Section 2 reviews fundamental definitions and lemmas that will be used throughout the paper. In Section 3, we establish a pointwise estimate for $V_{a}(e^{-tL^{\alpha}})$, which subsequently yields the proof of Theorem 1.1. Finally, Section 4 presents the proof of Theorem 1.2 via an extrapolation approach.\\
\indent We conclude the introduction with some conventions on the notation. We use $a\lesssim b$ to say that there exists a constant $C$, which is independent of the important parameters, such that $a\leq Cb$. Moreover, we write $a\sim b$ if $a\lesssim b$ and $b\lesssim a$. For any measurable set $E$, $|E|$ represents its Lebesgue measure. For a weight $w$, we set $w(E):=\int_{E}wdx$. Let $\chi_{E}$ stand for the characteristic function of $E$. In the paper all cubes are assumed to have edges parallel to the coordinate axes. 
\section{Preliminary}
\quad In this section, we recall some basic notations and present several lemmas that will be utilized in subsequent arguments.
\subsection{The auxiliary function and weight class}
\quad For $x\in\mathbb{R}^{n}$ and $V\in RH_{s}$ with $s>n/2$, where 
$$RH_{s}:=\left\{0\leq w\in L_{loc}:[w]_{RH_{s}}:=\sup_{Q}\frac{\langle w\rangle_{s,Q}}{\langle w\rangle_{Q}}<\infty\right\},$$
we define the auxiliary function associated with $V$ by
$$\rho(x):=\sup_{r>0}\left\{r:\frac{1}{r^{n-2}}\int_{B(x,r)}V(y)dy\leq1\right\}.$$
Shen \cite{S95} established that there exist constants $C,N>0$ such that for all $x,y\in\mathbb{R}^{n}$,
\begin{equation}C^{-1}\rho(x)\left(1+\frac{|x-y|}{\rho(x)}\right)^{-N}\leq\rho(y)\leq C\rho(x)\left(1+\frac{|x-y|}{\rho(x)}\right)^{\frac{N}{N+1}}.\tag{2.1}
\end{equation}
A function satisfies $(2.1)$ is called a critical radius function. In particular, $\rho(x)\sim\rho(y)$ whenever $|x-y|\lesssim\rho(x)$. For a cube $Q=Q(x_{0},r)$ and $\theta\geq0$, we denote $$\psi_{\theta}(Q):=\left(1+\frac{r}{\rho(x_{0})}\right)^{\theta}.$$
We shall make use of the following covering theorem associated with the critical radius function.
\begin{lem}
    \textnormal{(Dziuba\'{n}ski and Zinkiewicz \cite{DZ99})} Let $\rho$ be a critical radius function. There exists a sequence of points $x_{j}$ in $\mathbb{R}^{n}$, such that the family of balls $\{B_{j}:=B(x_{j},\rho(x_{j}))\}_{j\in\mathbb{Z}^{+}}$ satisfies the following properties:\\
    \textnormal{(i)} $\bigcup_{j\in\mathbb{Z}^{+}}B_{j}=\mathbb{R}^{n};$\\
    \textnormal{(ii)} For any $\sigma\geq1$, there exist constants $C,N>0$ such that for any $x\in\mathbb{R}^{n}$, $$\sum_{j\in\mathbb{Z}^{+}}\chi_{\sigma B_{j}}(x)\leq C\sigma^{N}.$$
\end{lem}
 A weight is a non-negative locally integrable function on $\mathbb{R}^{n}$. For $\theta\geq0$ and a critical radius function $\rho$, we say that $w\in A_{p}^{\rho,\theta}$ for $1<p<\infty$ if
$$[w]_{A_{p}^{\rho,\theta}}:=\sup_{Q}\left(\frac{1}{|Q|}\int_{Q}w(x)dx\right)\left(\frac{1}{|Q|}\int_{Q}w(x)^{1-p^{\prime}}dx\right)^{p-1}\psi_{\theta}(Q)^{-1}<\infty.$$
We say that $w\in A_{1}^{\rho,\theta}$ if
$$[w]_{A_{1}^{\rho,\theta}}:=\sup_{Q}\frac{1}{|Q|}\int_{Q}w(x)dx(\underset{x\in Q}{\text{essinf}}\ w(x))^{-1}\psi_{\theta}(Q)^{-1}<\infty.$$
We also define 
$$A_{p}^{\rho}:=\bigcup_{\theta\geq0}A_{p}^{\rho,\theta},\quad A_{\infty}^{\rho}:=\bigcup_{p\geq1}A_{p}^{\rho}.$$
In particular, the classes $A_{p}^{\rho,\theta}$ are increasing in $\theta$, and they coincide with the classical Muckenhoupt $A_{p}$ classes when $\theta=0$.

Extensive research has been studied on the properties of such weights. In this work, We shall use the following property.
\begin{lem}
    \textnormal{(Wang, Zhao, Li and Liu \cite{WZLL25})} Let $1\leq p<\infty$. If $w\in A_{p}^{\rho}$, then there exist constants $0<\delta<1$, $\eta>0$ and $C>0$ such that for any cube $Q$ and measurable set $E\subseteq Q$,
    $$\frac{w(E)}{w(Q)}\leq C\psi_{\eta}(Q)\left(\frac{|E|}{|Q|}\right)^{\delta}.$$    
\end{lem}
\begin{rem}
    In the cited reference, the result is stated for balls. However, we apply the analogous result for cubes here. The proof remains entirely unchanged.
\end{rem}
\subsection{Function spaces and Operators}
\quad We begin by the study of the Orlicz space. Recall that $\Phi:[0,\infty)\rightarrow[0,\infty)$ is called a Young function if it is continuous, increasing, convex and satisfies $\Phi(0)=0$ and $\lim_{t\rightarrow\infty}\Phi(t)/t=\infty$. The corresponding complementary function of $\Phi$, $\bar{\Phi}:[0,\infty)\rightarrow[0,\infty)$ is given by
$$\bar{\Phi}(t)=\sup_{s>0}\{st-\Phi(s)\}.$$
Let $\Phi$ be a Young function. The localized Orlicz norm $\|f\|_{\Phi,Q}$ is defined by 
$$\|f\|_{\Phi,Q}:=\textnormal{inf}\left\{\lambda>0:\frac{1}{|Q|}\int_{Q}\Phi\left(\frac{|f(x)|}{\lambda}\right)dx\leq1\right\}.$$
Especially when $\Phi(t)=t^{p}$, we simply write $\|\cdot\|_{p,Q}$. Given $1<p<\infty$, we say that a Young function $\Phi$ belongs to $B_{p}$ if there exists a constant $c>0$ such that
$$[\Phi]_{B_{p}}:=\int_{c}^{\infty}\frac{\Phi(t)}{t^{p+1}}dt<\infty.$$
 
Then we define the Hardy-Littlewood maximal operator in the setting of critical radius functions. For $\theta\geq0$, define
$$M^{\rho,\theta}f(x)=\sup_{Q\ni x}\frac{1}{|Q|\psi_{\theta}(Q)}\int_{Q}|f(y)|dy.$$
It satisfies the following weighted inequality.
\begin{lem}
    \textnormal{(Bui, Bui and Duong \cite{BBD22}; Bongioanni, Harboure and Quijano \cite{BHQ20})} Let $\rho$ be a critical radius function, $\theta\geq0$ and $1<p<\infty$. Then we have\\
    \textnormal{(i)} $\|M^{\rho,\theta}f\|_{L^{p}(w)}\lesssim[w]_{A_{p}^{\rho,\theta(p-1)}}^{1/(p-1)}\|f\|_{L^{p}(w)}$.\\
    \textnormal{(ii)} Let $\Psi$ be a Young function such that $\bar{\Psi}\in B_{p}$, and a pair $(u,v)$ of weights satisfies
    $$[u,v]_{\Psi,p,\rho,\theta}:=\sup_{Q}\|v^{\frac{1}{p}}\|_{p,Q}\|u^{-\frac{1}{p}}\|_{\Psi,Q}\psi_{\theta}(Q)^{-1}<\infty.$$
    Then here exists $\sigma_{0}$ (depend on $\theta$) such that for all $\sigma\geq\sigma_{0}$,
    $$\|M^{\rho,\sigma}f\|_{L^{p}(v)}\lesssim[u,v]_{\Psi,p,\rho,\theta}[\bar{\Psi}]_{B_{p}}^{\frac{1}{p}}\|f\|_{L^{p}(u)}.$$
\end{lem}
Next, we define the sparse operator associated with the critical radius function. We begin with the notion of the shifted dyadic cubes and the sparse collection. For $t\in\left\{0,1,2\right\}^{n}$, we denote the shifted dyadic cubes in $\mathbb{R}^{n}$ by 
$$\mathcal{D}^{t}:=\left\{2^{-k}\left([0,1)^{n}+m+(-1)^{k}\frac{t}{3}\right);\ k\in\mathbb{Z},m\in\mathbb{Z}^{n}\right\},\ \text{and}\ \mathcal{D}:=\bigcup_{t}\mathcal{D}^{t}.$$
For $Q_{0}\in\mathcal{D}^{t}$, we denote by $\mathcal{D}^{t}(Q_{0})$ the collection of all dyadic cubes $Q\in\mathcal{D}^{t}$ that satisfy $Q\subseteq Q_{0}$.
\begin{defi}
For $t\in\{0,1,2\}^{n}$, a collection of cubes $\mathcal{S}\subseteq\mathcal{D}^{t}$ is said to be a sparse collection, if there is a pairwise disjoint collection $\{E_Q\}_{Q\in\mathcal{S}}$, so that $E_Q\subseteq Q$ and $|E_Q|\geq|Q|/2$.
\end{defi}
Let $\rho$ be a critical radius function, $\sigma\geq0$ and $\mathcal{S}$ be a sparse family. Define
$$\mathcal{A}_{\mathcal{S}}^{\rho,\sigma}f(x):=\sum_{Q\in\mathcal{S}}\langle f\rangle_{3Q}\psi_{\sigma}(Q)^{-1}\chi_{Q}(x).$$
We briefly denote the sparse operator by $\mathcal{A}_{\mathcal{S}}$ in the special case $\sigma=0$. We have the following weighted estimate for such sparse operator.
\begin{lem}
    \textnormal{(Bui, Bui and Duong \cite{BBD21}; Wen and Wu \cite{WW25})} Let $\rho$ be a critical radius function, $\theta\geq0$ and $1<p<\infty$. Then we have\\
    \textnormal{(i)} For $\sigma\geq\theta\ \textnormal{max}\{1,1/(p-1)\}$,
$$\|\mathcal{A}_{\mathcal{S}}^{\rho,\sigma}f\|_{L^{p}(w)}\lesssim[w]_{A_{p}^{\rho,\theta}}^{\textnormal{max}\{1,\frac{1}{p-1}\}}\|f\|_{L^{p}(w)}.$$
\textnormal{(ii)} Let $\Psi$ be a Young function such that $\bar{\Psi}\in B_{p}$, then
$$\|\mathcal{A}_{\mathcal{S}}^{\rho,\theta}f\|_{L^{p}(v)}\lesssim[u,v]_{\Psi,p,\rho,\theta}[\bar{\Psi}]_{B_{p}}^{\frac{1}{p}}\|f\|_{L^{p}(u)}.$$
\end{lem}
\begin{rem}
Here $\textnormal{(ii)}$ is a modified version of the result by Wen and Wu \textnormal{\cite{WW25}}. In fact, the authors established that
$$\|\mathcal{A}_{\mathcal{S}}^{\rho,\theta}f\|_{L^{p}(v)}\lesssim[u,v]_{\Phi,\Psi,p,\rho,\theta}[\bar{\Phi}]_{B_{p^{\prime}}}^{\frac{1}{p^{\prime}}}[\bar{\Psi}]_{B_{p}}^{\frac{1}{p}}\|f\|_{L^{p}(u)},$$
with $\Phi,\Psi$ are Young functions that $\bar{\Phi}\in B_{p^{\prime}}$, $\bar{\Psi}\in B_{p}$ and a pair of weights $(u,v)$ satisfies
$$[u,v]_{\Phi,\Psi,p,\rho,\theta}:=\sup_{Q}\|v^{\frac{1}{p}}\|_{\Phi,Q}\|u^{-\frac{1}{p}}\|_{\Psi,Q}\psi_{\theta}(Q)^{-1}<\infty.$$
\textnormal{(ii)} can be proved using exactly the same method.
\end{rem}
We end this subsection by recalling a principle of sparse domination.
\begin{prop}
\textnormal{(Lerner \cite{L13})} Let $T$ be a sublinear operator which is of weak type \textnormal{($r,r$)} with $1\leq r<\infty$. Suppose that $M_{T}$ is of weak type \textnormal{($r,r$)}. Then there is a constant $C>0$ such that for all $f\in L^{r}$ supported on $3Q_{0}$ for some $Q_{0}\in\mathcal{D}$, there is a sparse family $\mathcal{S}\subseteq\mathcal{D}(Q_{0})$ (depending on $f$) such that
$$|Tf|\chi_{Q_{0}}\leq C\sum_{P\in\mathcal{S}}\langle|f|^{r}\rangle_{3P}^{\frac{1}{r}}\chi_{P}.$$
Here $\mathcal{D}(Q_{0})$ is the set of all dyadic cubes with respect to $Q_{0}$, and 
$$M_{T}f(x):=\sup_{Q\ni x}\|T(f\chi_{\mathbb{R}^{n}\backslash3Q})\|_{L^{\infty}(Q)},$$
where the supremum is taken over all cubes $Q\subseteq\mathbb{R}^{n}$ containing $x$.
\end{prop}
\subsection{Fractional Schr\"{o}dinger semigroups}
\quad According to the basic theory of semigroups, for $0\leq V\in RH_{s}$ with $s>q/2$, the Schr\"{o}dinger operator $L=-\Delta+V$ generates the Schr\"{o}dinger semigroup $\{e^{-tL}\}_{t>0}$, that is,
$$L(f):=\lim_{t\rightarrow0}\frac{f-e^{-tL}f}{t},$$
where the limit is taken in the sense of $L^{2}(\mathbb{R}^{n})$. It is known that the Schr\"{o}dinger semigroup $\{e^{-tL}\}_{t>0}$ admits an integral kernel $K_{t}^{L}(\cdot,\cdot)$ via the Freyman-Kac formula. Moreover, as established by Simon \cite{S82}, the kernel $K_{t}^{L}$ is a positive and symmetric function on $\mathbb{R}^{n}\times\mathbb{R}^{n}$. We mainly focus on the fractional Schr\"{o}dinger semigroup $\{e^{-tL^{\alpha}}\}$ with $0<\alpha<1$. It should be noted that when $V=0$, the fractional Schr\"{o}dinger semigroup $\{e^{-t(-\Delta)^{\alpha}}\}_{t>0}$ generated by $(-\Delta)^{\alpha}$ in $L^{2}$ with $0<\alpha<1$, can be equivalently defined via the Fourier transform as  
$$e^{-t(-\Delta)^{\alpha}}f:=\mathcal{F}^{-1}[e^{-t|\xi|^{2\alpha}}\hat{f}(\xi)].$$
This plays an important role in many fields of mathematics, such as harmonic analysis and PDEs. According to Grigor'yan \cite{G02}, we use the subordinative formula to express the integral kernel $K_{\alpha,t}^{L}(\cdot,\cdot)$ of $e^{-tL^{\alpha}}$ as
\begin{equation}
    K_{\alpha,t}^{L}(x,y)=\int_{0}^{\infty}\eta_{t}^{\alpha}(s)K_{s}^{L}(x,y)ds,\tag{2.2}
\end{equation}
where $\eta^{\alpha}_{t}(\cdot)$ is a continuous function on $(0,\infty)$. Building on the work of Li, Wang, Qian and Zhang \cite{LWQZ22}, we immediately obtain the following pointwise estimates.
\begin{lem}
    \textnormal{(Li, Wang, Qian and Zhang \cite{LWQZ22})} Let $0\leq V\in RH_{s}$ with $s>n/2$. Suppose $K_{\alpha,t}^{L}$ to be the integral kernel of $e^{-tL^{\alpha}}$ with $L=-\Delta+V$ and $0<\alpha<1$. Then we have\\
    \textnormal{(i)} For $N>0$, there exist constant $C_{N}>0$ such that for any $x,y\in\mathbb{R}^{n}$,
    $$|K_{\alpha,t}^{L}(x,y)|\leq C_{N}\ \textnormal{min}\left\{\frac{t^{1+\frac{N}{2\alpha}}}{|x-y|^{n+2\alpha+N}},t^{-\frac{n}{2\alpha}}\right\}\left(1+\frac{t^{\frac{1}{2\alpha}}}{\rho(x)}+\frac{t^{\frac{1}{2\alpha}}}{\rho(y)}\right)^{-N}.$$
    Further more,
    $$|K_{\alpha,t}^{L}(x,y)|\leq\frac{C_{N}t}{\left(t^{\frac{1}{2\alpha}}+|x-y|\right)^{n+2\alpha}}\left(1+\frac{t^{\frac{1}{2\alpha}}}{\rho(x)}+\frac{t^{\frac{1}{2\alpha}}}{\rho(y)}\right)^{-N}.$$
    \textnormal{(ii)} For $N>0$ and $m\in\mathbb{Z}^{+}$, there exist constant $C_{N,m}$ such that for any $x,y\in\mathbb{R}^{n}$,
    $$|\partial^{m}_{t}K_{\alpha,t}^{L}(x,y)|\leq C_{N,m}\ \textnormal{min}\left\{\frac{t^{1+\frac{N}{2\alpha}-m}}{|x-y|^{n+2\alpha+N}},t^{m-\frac{n}{2\alpha}}\right\}\left(1+\frac{t^{\frac{1}{2\alpha}}}{\rho(x)}+\frac{t^{\frac{1}{2\alpha}}}{\rho(y)}\right)^{-N}.$$
    Further more,
    $$|\partial_{t}^{m}K_{\alpha,t}^{L}(x,y)|\leq\frac{C_{N,m}}{\left(t^{\frac{1}{2\alpha}}+|x-y|\right)^{n+2\alpha m}}\left(1+\frac{t^{\frac{1}{2\alpha}}}{\rho(x)}+\frac{t^{\frac{1}{2\alpha}}}{\rho(y)}\right)^{-N}.$$
\end{lem}
\begin{rem}
    \textnormal In the cited paper, the authors state only the second result in both \textnormal{(i)} and \textnormal{(ii)}. The first result in \textnormal{(i)} and \textnormal{(ii)} can be derived with minor modifications to their proof.
\end{rem}
\section{Proof of Theorem 1.1}
\quad In this section, we present the proof of Theorem 1.1. Our approach is primarily inspired by the work of Wen and Wu \cite{WW25}. The argument proceeds by establishing a pointwise estimate for the variation operator $V_{a}(e^{-tL^{\alpha}})$. Let $\{B_{j}\}$ be a family of critical balls that satisfy the conditions in Lemma 2.1, it is straightforward to show that there exist dyadic cubes $Q_{j}\in\mathcal{D}$ such that 
\begin{equation}
B_{j}\subseteq Q_{j}\subseteq\alpha_{0}B_{j}\tag{3.1}
\end{equation}
for some constant $\alpha_{0}>0$. We claim that
\begin{lem}
Let $\{B_{j}\}$ be a family of critical balls in Lemma 2 and let dyadic cubes $\{Q_{j}\}$ satisfy \textnormal{(3.1)}. Then there exist sparse collection $\mathcal{S}_{j}\subseteq\mathcal{D}(Q_{j})$ such that for any $\gamma,\tau>0$, we have
$$V_{a}(e^{-tL^{\alpha}})f(x)\lesssim\sum_{j}\mathcal{A}_{\mathcal{S}_{j}}^{\rho,\tau}(f\chi_{3Q_{j}})(x)+M^{\rho,\gamma}f(x).$$
\end{lem}

According to the definition of $Q_{j}$, we have $\sum_{j}\chi_{3Q_{j}}(x)\leq C$ for any $x\in\mathbb{R}^{n}$. Thus, Theorem 1.1 follows directly from Lemma 3.1, by applying Lemma 2.3 and Lemma 2.5 (here we choose $\gamma=\theta/(p-1)$ to apply Lemma 2.3). Therefore, it remains to prove Lemma 3.1. To this end, we require the following lemmas.
\begin{lem}
    \textnormal{(Auscher \cite{A07})} Let $1\leq p_{0}<2$. Suppose that $T$ is a sublinear operator of strong type \textnormal{$(2,2)$} and let $A_{r}\ (r>0)$ be a family of linear operators acting on $L^{2}$. For a ball $B$, define $C_{1}(B)=4B$ and $C_{j}(B)=2^{j+1}B\ \backslash2^{j}B$ for $j\geq2$. Assume for $j\geq2$,
    \begin{equation}
        \left(\frac{1}{|2^{j+1}B|}\int_{C_{j}(B)}|T(I-A_{r(B)})f(x)|^{2}dx\right)^{\frac{1}{2}}\leq g(j)\left(\frac{1}{|B|}\int_{B}|f(x)|^{p_{0}}dx\right)^{\frac{1}{p_{0}}}\tag{3.2}
    \end{equation}
and for $j\geq1$
\begin{equation}
    \left(\frac{1}{|2^{j+1}B|}\int_{C_{j}(B)}|A_{r(B)}f(x)|^{2}dx\right)^{\frac{1}{2}}\leq g(j)\left(\frac{1}{|B|}\int_{B}|f(x)|^{p_{0}}dx\right)^{\frac{1}{p_{0}}}\tag{3.3}
\end{equation}
for all Ball $B$ and all $f$ supported on $B$. If $\Sigma=\sum_{j\geq1}g(j)2^{nj}<\infty$, then $T$ is of weak type \textnormal{$(p_{0},p_{0})$} with a bound depending only on the strong type \textnormal{$(2,2)$} bound of $T$, $p_{0}$ and $\Sigma$.
\end{lem}
\begin{lem}
Let $K_{\alpha,t}^{L}$ be the integral kernel of $e^{-tL^{\alpha}}$, then for any cube $Q$, $\xi,y\in Q$ and function $g$, it holds that
$$\int_{\mathbb{R}^{n}}|K_{\alpha,r_{Q}^{2\alpha}}^{L}(\xi,y)||g(\xi)|d\xi\lesssim M(g)(x),$$
where $M$ denotes the classical Hardy-Littlewood maximal operator.
\end{lem}
\begin{proof}[Proof of Lemma 3.3.]
By using Lemma 2.7, We estimate that    
\begin{align*}
    \int_{\mathbb{R}^{n}}|K_{\alpha,r_{Q}^{2\alpha}}^{L}(\xi,y)||g(\xi)|d\xi&\lesssim\sum_{j=2}^{\infty}\int_{2^{j}Q\backslash2^{j-1}Q}\frac{r_{Q}^{2\alpha}|g(\xi)|}{r_{Q}^{2\alpha\left(1+\frac{n}{2\alpha}\right)}+|\xi-y|^{n+2\alpha}}d\xi+\int_{2Q}r_{Q}^{-n}|g(\xi)|d\xi\\
    &\lesssim\sum_{j=2}^{\infty}\int_{2^{j}Q}\frac{r_{Q}^{2\alpha}}{r_{Q}^{2\alpha+n}+2^{(j-3)(n+2\alpha)}r_{Q}^{2\alpha+n}}|g(\xi)|dx\xi+Mg(x)\\
    &\lesssim\sum_{j=2}^{\infty}2^{-2\alpha j}\frac{1}{|2^{j}Q|}\int_{2^{j}Q}|g(\xi)|d\xi+Mg(x)\\
    &\lesssim\sum_{j=2}^{\infty}2^{-2\alpha j}Mg(x)+Mg(x)\sim Mg(x),
\end{align*}
which completes the proof of the lemma.
\end{proof}
\begin{proof}[Proof of Lemma 3.1.] By the definition of $\{Q_{j}\}_{j}$, we can see that $\bigcup_{j}Q_{j}=\mathbb{R}^{n}$ and $\sum_{j}\chi_{Q_{j}}(x)\leq C$. Thus, for any $x\in\mathbb{R}^{n}$ we obtain that
\begin{align*}
    V_{a}(e^{-tL^{\alpha}})f(x)&\leq\sum_{j} V_{a}(e^{-tL^{\alpha}})f(x)\chi_{B_{j}}(x)\\
    &\leq\sum_{j}V_{a}(e^{-tL^{\alpha}})(f\chi_{\mathbb{R}^{n}\backslash3Q_{j}})(x)\chi_{B_{j}}(x)+\sum_{j}V_{a}(e^{-tL^{\alpha}})(f\chi_{3Q_{j}})(x)\chi_{B_{j}}(x)\\
    &=:\textnormal{I}_{1}+\textnormal{I}_{2}.\tag{3.4}
\end{align*}

We begin by estimating $\textnormal{I}_{1}$. From the functional calculus in $L^{2}$, it follows that the integral kernel of the operator $L^{\alpha}e^{-tL^{\alpha}}$ is given by $-\partial_{t}K_{\alpha,t}^{L}$. According to the definition of $\{B_{j}\}$ in Lemma 2.1, for any $x\in B_{j}$, we have $\rho(x)\sim\rho(x_{j})$. Since $|y-x_{j}|>\rho(x_{j})$ for any $y\in\mathbb{R}^{n}\backslash3B_{j}$, it follows from Lemma 2.7 that
\begin{align*}
    V_{a}(e^{-tL^{\alpha}})(f\chi_{\mathbb{R}^{n}\backslash3Q_{j}})(x)&\leq\sup_{\{t_{j}\}\searrow0}\sum_{k}\left|\int_{t_{k+1}}^{t_{k}}L^{\alpha}e^{-tL^{\alpha}}(f\chi_{\mathbb{R}^{n}\backslash3Q_{j}})(x)dt\right|\\
    &\leq\int_{0}^{\infty}\left|L^{\alpha}e^{-tL^{\alpha}}(f\chi_{\mathbb{R}^{n}\backslash3Q_{j}})(x)\right|dt\\
    &\leq\int_{0}^{\infty}\int_{\mathbb{R}^{n}\backslash3B_{j}}|\partial_{t}K_{\alpha,t}^{L}(x,y)||f(y)|dydt\\
    &\lesssim\int_{0}^{\rho(x_{j})^{2\alpha}}\int_{\mathbb{R}^{n}\backslash3B_{j}}\min\left\{\frac{t^{\frac{N_{1}}{2\alpha}}}{|x-y|^{n+2\alpha+N_{1}}},t^{-\frac{n}{2\alpha}-1}\right\}|f(y)|dydt+\\
    &\quad\int_{\rho(x_{j})^{2\alpha}}^{|x_{j}-y|^{2\alpha}}\int_{\mathbb{R}^{n}\backslash3B_{j}}\frac{t^{\frac{N_{2}}{2\alpha}}}{|x-y|^{n+2\alpha+N_{2}}}\left(1+\frac{t^{\frac{1}{2\alpha}}}{\rho(x_{j})}\right)^{-N_{2}}|f(y)|dydt+\\
    &\quad\int_{|x_{j}-y|^{2\alpha}}^{\infty}\int_{\mathbb{R}^{n}\backslash3B_{j}}t^{-\frac{n}{2\alpha}-1}\left(1+\frac{t^{\frac{1}{2\alpha}}}{\rho(x_{j})}\right)^{-N_{3}}|f(y)|dydt\\
    &=:\textnormal{I}_{11}+\textnormal{I}_{12}+\textnormal{I}_{13},
\end{align*}
where $N_{i}$ $(i=1,2,3)$ will be determined later. Before we estimate $I_{11}$, $I_{12}$, $I_{13}$. We first establish an auxiliary estimate that will be used later. For any $\sigma>0$, we estimate that
\begin{align*}
    &\quad\ \sum_{k=1}^{\infty}\frac{1}{|3^{k+1}B_{j}|}\int_{3^{k+1}B_{j}}|f(y)|\left(\frac{\rho(x_{j})}{3^{k}\rho(x_{j})}\right)^{\sigma+1}dy\\&\lesssim\sum_{k=1}^{\infty}3^{-k}\left(1+\frac{3^{k+1}\rho(x_{j})}{\rho(x_{j})}\right)^{-\sigma}\frac{1}{|3^{k+1}B_{j}|}\int_{3^{k+1}B_{j}}|f(y)|dy\\
    &\lesssim M^{\rho,\sigma}f(x).\tag{3.5}
\end{align*}
We begin by estimating $\textnormal{I}_{11}$. For any $N_{1}>0$ and $x\in B_{j}$, by (3.5) we obtain that
\begin{align*}
    \textnormal{I}_{11}&\leq\int_{\mathbb{R}^{n}\backslash3B_{j}}\int_{0}^{\rho(x_{j})^{2\alpha}}\frac{t^{-\frac{1}{2}+\frac{N_{1}}{2\alpha}}}{|x-y|^{n+\alpha+N_{1}}}|f(y)|dtdy\\
    &\sim\sum_{k=1}^{\infty}\int_{3^{k+1}B_{j}\backslash3^{k}B_{j}}\frac{|f(y)|}{|x-y|^{n}}\left(\frac{\rho(x_{j})}{|x-y|}\right)^{N_{1}+\alpha}dy\\
    &\lesssim\sum_{k=1}^{\infty}\frac{1}{|3^{k+1}B_{j}|}\int_{3^{k+1}B_{j}}|f(y)|\left(\frac{\rho(x_{j})}{3^{k}\rho(x_{j})}\right)^{N_{1}+\alpha}\lesssim M^{\rho,N_{1}+\alpha-1}f(x),\tag{3.6}
\end{align*}
where the first inequality follows from a straightforward estimate
$$\min\left\{\frac{t^{\frac{N_{1}}{2\alpha}}}{|x-y|^{n+2\alpha+N_{1}}},t^{-\frac{n}{2\alpha}-1}\right\}\leq\frac{t^{-\frac{1}{2}+\frac{N_{1}}{2\alpha}}}{|x-y|^{n+\alpha+N_{1}}}.$$
Since $0<\alpha<1$, we may choose $N_{1}=\gamma+1-\alpha>0$. We then proceed to estimate $\textnormal{I}_{12}$. For any $N_{2}>0$ and $x\in B_{j}$, it follows again from (3.5) that
\begin{align*}
    \textnormal{I}_{12}&\lesssim\int_{\mathbb{R}^{n}\backslash3B_{j}}\int_{\rho(x_{j})^{2\alpha}}^{|x_{j}-y|^{2\alpha}}\frac{t^{-\frac{1}{2}+\frac{N_{2}}{2\alpha}}}{|x-y|^{n+\alpha+N_{2}}}\left(\frac{\rho(x_{j})}{t^{\frac{1}{2\alpha}}}\right)^{N_{2}}|f(y)|dtdy\\
    &\lesssim\int_{\mathbb{R}^{n}\backslash3B_{j}}\frac{|f(y)|}{|x-y|^{n}}\left(\frac{\rho(x_{j})}{|x-y|}\right)^{N_{2}}\left(\frac{|x_{j}-y|}{|x-y|}\right)^{\alpha}dy\\
    &\lesssim\sum_{k=1}^{\infty}\int_{3^{k+1}B_{j}\backslash3^{k}B_{j}}\frac{|f(y)|}{|x-y|^{n}}\left(\frac{\rho(x_{j})}{|x-y|}\right)^{N_{2}}dy\\
    &\lesssim\sum_{k=1}^{\infty}\frac{1}{|3^{k+1}B_{j}|}\int_{3^{k+1}B_{j}}|f(y)|\left(\frac{\rho(x_{j})}{3^{k}\rho(x_{j})}\right)^{N_{2}}dy\lesssim M^{\rho,N_{2}-1}f(x).\tag{3.7}
\end{align*}
We choose $N_{2}=\gamma+1$. We finally turn to the estimation of $\textnormal{I}_{13}$. For $N_{3}>0$ and $x\in B_{j}$, by (3.5), we have that
\begin{align*}
    \textnormal{I}_{13}&\leq\int_{\mathbb{R}^{n}\backslash3B_{j}}\int_{|x_{j}-y|^{2\alpha}}^{\infty}t^{-\frac{n}{2\alpha}-1}\left(\frac{\rho(x_{j})}{t^{\frac{1}{2\alpha}}}\right)^{N_{3}}|f(y)|dtdy\\
    &\sim\int_{\mathbb{R}^{n}\backslash3B_{j}}\frac{|f(y)|}{|x_{j}-y|^{n}}\left(\frac{\rho(x_{j})}{|x_{j}-y|}\right)^{N_{3}}dy\\
    &\lesssim\sum_{k=0}^{\infty}\frac{1}{|3^{k+1}B_{j}|}\int_{3^{k+1}B_{j}}|f(y)|\left(\frac{\rho(x_{j})}{3^{k}\rho(x_{j})}\right)^{N_{3}}dy\lesssim M^{\rho,N_{3}-1}f(x).\tag{3.8}
\end{align*}
We again choose $N_{3}=\gamma+1$. By combining (3.6), (3.7) and (3.8), we conclude that 
\begin{equation}
\textnormal{I}_{1}\lesssim M^{\rho,\gamma}f(x).\tag{3.9}
\end{equation}

For the estimation of \textnormal{II}, we intend to apply Proposition 2.6 to obtain that for any $j$, there exist sparse family $\mathcal{S}_{j}\subseteq\mathcal{D}(Q_{j})$ such that
$$V_{a}(e^{-tL^{\alpha}})(f\chi_{3Q_{j}})(x)\chi_{Q_{j}}(x)\lesssim\mathcal{A}_{\mathcal{S}_{j}}(f\chi_{3Q_{j}})(x).$$
Wang, Zhao, Li and Liu \cite{WZLL25} established the weak type $(1,1)$ boundedness of the operator $M_{V_{a}(e^{-tL^{\alpha}})}$. By functional calculus, it is immediately observed that
$$P_{t}(L^{\alpha}):=\int_{t}^{\infty}L^{\alpha}e^{-sL^{\alpha}}ds=e^{-tL^{\alpha}}.$$
We further denote
$$S_{t}(L^{\alpha}):=I-P_{t}(L^{\alpha})=\int_{0}^{\infty}L^{\alpha}e^{-sL^{\alpha}}ds-\int_{t}^{\infty}L^{\alpha}e^{-sL^{\alpha}}ds=\int_{0}^{t}L^{\alpha}e^{-sL^{\alpha}}ds.$$
For any $x\in\mathbb{R}^{n}$, we estimate that
\begin{align*}
    M_{V_{a}(e^{-tL^{\alpha}})}f(x)&=\sup_{Q\ni x}\|V_{a}(e^{-tL^{\alpha}})(f\chi_{\mathbb{R}^{n}\backslash3Q_{}})\|_{L^{\infty}(Q)}\\
    &\leq\sup_{Q\ni x}\|V_{a}(e^{-tL^{\alpha}})P_{r_{Q}^{2\alpha}}(L^{\alpha})f\|_{L^{\infty}(Q)}+\sup_{Q\ni x}\|V_{a}(e^{-tL^{\alpha}})P_{r_{Q}^{2\alpha}}(L^{\alpha})(f\chi_{3Q})\|_{L^{\infty}(Q)}+\\
    &\ \ \ \sup_{Q\ni x}\|V_{a}(e^{-tL^{\alpha}})S_{r_{Q}^{2\alpha}}(L^{\alpha})(f\chi_{\mathbb{R}^{n}\backslash3Q})\|_{L^{\infty}(Q)}\\
    &=:M_{1}f(x)+M_{2}f(x)+M_{3}f(x).\tag{3.10}
\end{align*}
We now show that the operator $M_{V_{a}(e^{-tL^{\alpha}})}$ is of weak type (1,1). We begin by estimating $M_{2}$ and $M_{3}$. First, we establish that
\begin{align*}
    \|V_{a}(e^{-tL^{\alpha}})P_{r_{Q}^{2\alpha}}(L^{\alpha})(f\chi_{3Q})\|_{L^{\infty}(Q)}&=\underset{y\in Q}{\textnormal{esssup}}\sup_{\{t_{j}\}\searrow0}\sum_{i=1}^{\infty}\left|\int_{t_{i+1}}^{t_{i}}L^{\alpha}e^{-tL^{\alpha}}P_{r_{Q}^{2\alpha}}(L^{\alpha})(f\chi_{3Q})(y)dt\right|\\
    &\leq\underset{y\in Q}{\textnormal{esssup}}\int_{0}^{\infty}\left|L^{\alpha}e^{-tL^{\alpha}}P_{r_{Q}^{2\alpha}}(L^{\alpha})(f\chi_{3Q})(y)dt\right|\\
    &=\underset{y\in Q}{\textnormal{esssup}}\int_{0}^{\infty}\left|L^{\alpha}e^{-tL^{\alpha}}\int_{r_{Q}^{2\alpha}}^{\infty}L^{\alpha}e^{-tL^{\alpha}}(f\chi_{3Q})(y)ds\right|dt\\
    &\leq\underset{y\in Q}{\textnormal{esssup}}\int_{0}^{\infty}\int_{r_{Q}^{2\alpha}}^{\infty}\left|(L^{\alpha})^{2}e^{-(s+t)L^{\alpha}}(f\chi_{3Q})(y)\right|dsdt.
\end{align*}
Since the integral kernel of the operator $(L^{\alpha})^{2}e^{-tL^{\alpha}}$ is given by $\partial^{2}_{t}K_{\alpha,t}^{L}$, it follows from Lemma 2.7 that
$$\left|(L^{\alpha})^{2}e^{-(s+t)L^{\alpha}}(f\chi_{3Q})(y)\right|\lesssim\left|\int_{3Q}\frac{f(z)}{((t+s)^{\frac{1}{2\alpha}}+|z-y|)^{n+4\alpha}}dz\right|\leq\int_{3Q}(t+s)^{-\frac{n}{2\alpha}-2}|f(z)|dz.$$
Hence, it follows that
\begin{align*}
    M_{2}f(x)&\lesssim\sup_{Q\ni x}\int_{0}^{\infty}\int_{r_{Q}^{2\alpha}}^{\infty}\int_{3Q}(t+s)^{-\frac{n}{2\alpha}-2}|f(z)|dzdsdt\\
    &\sim\sup_{Q\ni x}\  r_{Q}^{n}\langle f\rangle_{3Q}\int_{r_{Q}^{2\alpha}}^{\infty}s^{-\frac{n}{2\alpha}-1}ds\\
    &\sim\sup_{Q\ni x}\ \langle f\rangle_{3Q}\leq Mf(x),\tag{3.11}
\end{align*}
Thus, $M_{2}$ is of weak type $(1,1)$. Similarly, for $M_{3}$, we have
\begin{align*}
    \|V_{a}(e^{-tL^{\alpha}})S_{r_{Q}^{2\alpha}}(L^{\alpha})(f\chi_{\mathbb{R}^{n}\backslash3Q})\|_{L^{\infty}(Q)}
    &=\underset{y\in Q}{\textnormal{esssup}}\sup_{\{t_{j}\}\searrow0}\sum_{i=1}^{\infty}\left|\int_{t_{i+1}}^{t_{i}}L^{\alpha}e^{-tL^{\alpha}}S_{r_{Q}^{2\alpha}}(L^{\alpha})(f\chi_{\mathbb{R}^{n}\backslash3Q})(y)dt\right|\\
    &\leq\underset{y\in Q}{\textnormal{esssup}}\int_{0}^{\infty}\left|L^{\alpha}e^{-tL^{\alpha}}S_{r_{Q}^{2\alpha}}(L^{\alpha})(f\chi_{\mathbb{R}^{n}\backslash3Q})(y)dt\right|\\
    &\leq\underset{y\in Q}{\textnormal{esssup}}\int_{0}^{\infty}\int_{0}^{r_{Q}^{2\alpha}}\left|(L^{\alpha})^{2}e^{-(s+t)L^{\alpha}}(f\chi_{\mathbb{R}^{n}\backslash3Q})(y)\right|dsdt.
\end{align*}
Since $y\in Q$, again by using Lemma 2.7, we obtain that
\begin{align*}
    \int_{0}^{\infty}\int_{0}^{r_{Q}^{2\alpha}}\left|(L^{\alpha})^{2}e^{-(s+t)L^{\alpha}}(f\chi_{\mathbb{R}^{n}\backslash3Q})(y)\right|dsdt&\lesssim\int_{0}^{\infty}\int_{0}^{r_{Q}^{2\alpha}}\int_{|z-y|\geq r_{Q}}\frac{|f(z)|}{((t+s)^{\frac{1}{2\alpha}}+|z-y|)^{n+4\alpha}}dzdsdt\\
    &\lesssim\sum_{k=0}^{\infty}\int_{0}^{\infty}\int_{0}^{r_{Q}^{
    2\alpha}}\int_{2^{k}r_{Q}\leq|z-y|<2^{k+1}r_{Q}}\frac{|f(z)|}{(t+s+(2^{k}r_{Q})^{2\alpha})^{\frac{n}{2\alpha}+2}}dzdsdt\\
    &\lesssim\sum_{k=0}^{\infty}\int_{0}^{r_{Q}^{2\alpha}}\int_{2^{k+3}r_{Q}}\frac{|f(z)|}{(s+(2^{k}r_{Q})^{2\alpha})^{\frac{n}{2\alpha}+1}}dzds\\
    &\lesssim \sum_{k=0}^{\infty}r_{Q}^{2\alpha}(2^{k}r_{Q})^{-2\alpha\left(\frac{n}{2\alpha}+1\right)}(2^{k}r_{Q})^{n}\langle f\rangle_{2^{k+3}Q}\\
    &\sim\sum_{k=0}^{\infty}2^{-2\alpha k}\langle f\rangle_{2^{k+3}Q}\sim\langle f\rangle_{2^{k+1}Q}.
\end{align*}
Hence, it follows that
\begin{equation}
    M_{3}f(x)\lesssim\sup_{Q\ni x}\ \langle f\rangle_{2^{k+3}Q}\leq Mf(x).\tag{3.12}
\end{equation}
Thus, $M_{3}$ is also of weak type (1,1). It remains only to show $M_{1}$ is likewise of weak type (1,1). By using Lemma 3.3 and the Minkovski inequality ($a>2$), we observe that for any $x,\xi\in Q$, where $Q$ is a cube, it holds that
\begin{align*}
    V_{a}(e^{-tL^{\alpha}})P_{r_{Q}^{2\alpha}}(L^{\alpha})f(\xi)&=\sup_{\{t_{j}\}\searrow0}\left(\sum_{i=1}^{\infty}\left|e^{-(t_{i}+r_{Q}^{2\alpha})L^{\alpha}}f(\xi)-e^{-(t_{i+1}+r_{Q}^{2\alpha})L^{\alpha}}f(\xi)\right|^{a}\right)^{\frac{1}{a}}\\
    &=\sup_{\{t_{j}\}\searrow0}\left(\sum_{i=1}^{\infty}\left|e^{-r_{Q}^{2\alpha}L^{\alpha}}\left[e^{-t_{i}L^{\alpha}}f(\xi)-e^{-t_{i+1}L^{\alpha}}f(\xi)\right]\right|^{a}\right)^{\frac{1}{a}}\\
    &=\sup_{\{t_{j}\}\searrow0}\left(\sum_{i=1}^{\infty}\left|\int_{\mathbb{R}^{n}}K_{\alpha,r_{Q}^{2\alpha}}^{L}(\xi,y)\left[e^{-t_{i}L^{\alpha}}f(\xi)-e^{-t_{i+1}L^{\alpha}}f(\xi)\right]d\xi\right|^{a}\right)^{\frac{1}{a}}\\
    &\leq\int_{\mathbb{R}^{n}}|K_{\alpha,r_{Q}^{2\alpha}}^{L}(\xi,y)|\sup_{\{t_{j}\}\searrow0}\left(\sum_{i=1}^{\infty}\left|e^{-t_{i}L^{\alpha}}f(\xi)-e^{-t_{i+1}L^{\alpha}}f(\xi)\right|^{a}\right)^{\frac{1}{a}}d\xi\\
    &=\int_{\mathbb{R}^{n}}|K_{\alpha,r_{Q}^{2\alpha}}^{L}(\xi,y)|V_{a}(e^{-tL})f(\xi)d\xi\\
    &\lesssim M(V_{a}(e^{-tL^{\alpha}})f)(x).
\end{align*}
It follows that 
$$M_{1}f(x)\lesssim M(V_{a}(e^{-tL^{\alpha}})f)(x).$$
Since Wang, Zhao, Li and Liu \cite{WZLL25} established that ${V_{a}(e^{-tL^{\alpha}})}$ is bounded on $L^{2}$ in the nonweighted case, we conclude that $M_{1}$ is bounded on $L^{2}$. By applying Lemma 3.2, it suffices to verify inequalities (3.2) and (3.3) with $p_{0}=1$ and $A_{r(B)}=P_{r_{B}^{2\alpha}}(L^{\alpha})=e^{-r_{B}^{2\alpha}L^{\alpha}}$ for any $f$ supported on a ball $B$.

We begin by verifying inequality (3.3). For the case $j\geq2$, Lemma 2.7 yields that
\begin{align*}
    &\quad\ \left(\frac{1}{|2^{j+1}B|}\int_{2^{j+1}B\backslash2^{j}B}|P_{r_{B}^{2\alpha}}(L^{\alpha})f(x)|^{2}dx\right)^{\frac{1}{2}}\\
    &\lesssim\left(\frac{1}{2^{j+1}B}\int_{2^{j+1}B\backslash2^{j}B}\left[\int_{B}\frac{r_{B}^{2\alpha}|f(y)|}{r_{B}^{2\alpha(1+\frac{n}{2\alpha})}+|x-y|^{n+2\alpha}}dy\right]^{2}dx\right)^{\frac{1}{2}}\\
    &\lesssim\left[\frac{1}{|2^{j+1}B|}\int_{2^{j+1}B\backslash2^{j}B}\left(\frac{1}{|B|}\int_{B}2^{-(j-1)(n+2\alpha)}|f(y)|dy\right)^{2}\right]^{\frac{1}{2}}\\
    &\lesssim2^{-j(n+2\alpha)}\frac{1}{|B|}\int_{B}|f(y)|dy.\tag{3.13}
\end{align*}
For the case $j=1$, we only use $|x-y|\geq0$ to deduce that the left-hand side satisfies $LHS\leq\langle f\rangle_{B}$. We now proceed to verify inequality (3.2). For $j\geq2$ and $x\in2^{j+1}B\backslash2^{j}B$, we can see that
\begin{align*}
     M_{1}(S_{r_{B}^{2\alpha}}(L^{\alpha})f)(x)&=\sup_{Q\ni x}\underset{y\in Q}{\textnormal{esssup}}\ V_{a}(e^{-tL^{\alpha}})(P_{r_{Q}^{2\alpha}}(L^{\alpha})S_{r_{B}^{2\alpha}}(L^{\alpha})f)(y)\\
     &=\sup_{Q\ni x}\underset{y\in Q}{\textnormal{esssup}}\sup_{\{t_{j}\}\searrow0}\left(\sum_{i=1}^{\infty}\left|e^{-t_{i}L^{\alpha}}P_{r_{Q}^{2\alpha}}(L^{\alpha})S_{r_{B}^{2\alpha}}(L^{\alpha})f(y)-e^{-t_{i+1}L^{\alpha}}P_{r_{Q}^{2\alpha}}(L^{\alpha})S_{r_{B}^{2\alpha}}(L^{\alpha})f(y)\right|^{a}\right)^{\frac{1}{a}}\\
     &\leq\sup_{Q\ni x}\underset{y\in Q}{\textnormal{esssup}}\sup_{\{t_{j}\}\searrow0}\sum_{i=1}^{\infty}\left|\int_{t_{i+1}}^{t_{i}}L^{\alpha}e^{-tL^{\alpha}}P_{r_{Q}^{2\alpha}}(L^{\alpha})S_{r_{B}^{2\alpha}}(L^{\alpha})f(y)dt\right|\\
     &\leq\sup_{Q\ni x}\underset{y\in Q}{\textnormal{esssup}}\int_{0}^{\infty}\left|L^{\alpha}e^{-tL^{\alpha}}P_{r_{Q}^{2\alpha}}(L^{\alpha})S_{r_{B}^{2\alpha}}(L^{\alpha})f(y)dt\right|\\
     &\leq\sup_{Q\ni x}\underset{y\in Q}{\textnormal{esssup}}\int_{0}^{\infty}\int_{r_{Q}^{2\alpha}}^{\infty}\int_{0}^{r_{B}^{2\alpha}}\left|(L^{\alpha})^{3}e^{-(s+u+t)L^{\alpha}}f(y)\right|dsdudt.\tag{3.14}
\end{align*}
Since the integral kernel of $(L^{\alpha})^{3}e^{-tL^{\alpha}}$ is given by $-\partial^{3}_{t}K_{\alpha,t}^{L}$, it follows from Lemma 2.7 that
\begin{equation}
    (3.14)\lesssim\int_{0}^{\infty}\int_{r_{Q}^{2\alpha}}^{\infty}\int_{0}^{r_{B}^{2}}\int_{B}\frac{|f(z)|}{((s+u+t)^{\frac{1}{2\alpha}}+|z-y|)^{n+6\alpha}}dzdsdudt.\tag{3.15}
\end{equation}
Here, it suffices to consider cubes $Q$ such that $Q\cap2^{j+1}B\backslash2^{j}B\neq\varnothing$, since $x\in Q\cap B$. Therefore, for any $z\in B$ and $y\in Q$, one has
$$|y-z|\geq|x-z|-|y-x|\geq2^{j}r_{B}-\sqrt{n}r_{Q}.$$
\underline{For the case $2^{j}r_{B}\leq2\sqrt{n}r_{Q}$}, we have $u\geq r_{Q}^{2\alpha}\geq(2\sqrt{n})^{-2\alpha}2^{2\alpha j}r_{B}^{2\alpha}\gtrsim s$ in (3.15), with the implicit constant independent of $j$. Thus, we have
\begin{align*}
    (3.15)&\leq r_{B}^{n}\langle f\rangle_{B}\int_{0}^{\infty}\int_{r_{Q}^{2\alpha}}^{\infty}\int_{0}^{r_{B}^{2\alpha}}\frac{1}{(s+u+t)^{\frac{n}{2\alpha}+3}}dsdudt\\
    &\sim r_{B}^{n}\langle f\rangle_{B}\int_{r_{Q}^{2\alpha}}^{\infty}\int_{0}^{r_{B}^{2\alpha}}\frac{1}{(s+u)^{\frac{n}{2\alpha}+2}}dsdu\\
    &\sim r_{B}^{n}\langle f\rangle_{B}\int_{r_{Q}^{2\alpha}}^{\infty}\int_{0}^{r_{B}^{2\alpha}}u^{-\frac{n}{2\alpha}-2}dsdu\\
    &\sim r_{B}^{n+2\alpha}r_{Q}^{-n-2\alpha}\langle f\rangle_{Q}\lesssim2^{-j(n+2\alpha)}\langle f\rangle_{B}.\tag{3.16}
\end{align*}
\underline{For the case $2^{j}r_{B}\geq2\sqrt{n}r_{Q}$}, it follows that $|y-z|\geq2^{j-1}r_{B}$. One obtains that
\begin{align*}
    (3.15)&\lesssim r_{B}^{n}\langle f\rangle_{B}\int_{0}^{\infty}\int_{r_{Q}^{2\alpha}}^{\infty}\int_{0}^{r_{B}^{2\alpha}}\frac{1}{(t+s+u+(2^{j}r_{B})^{2\alpha})^{\frac{n}{2\alpha}+3}}dsdudt\\
    &\sim r_{B}^{n}\langle f\rangle_{B}\int_{r_{Q}^{2\alpha}}^{\infty}\int_{0}^{r_{B}^{2\alpha}}\frac{1}{(s+u+(2^{j}r_{B})^{2\alpha})^{\frac{n}{2\alpha}+2}}dsdu\\
    &\lesssim r_{B}^{n}\langle f\rangle_{B}\int_{r_{Q}^{2\alpha}}^{(2^{j}r_{B})^{2\alpha}}\int_{0}^{r_{B}^{2\alpha}}\frac{1}{(s+u+(2^{j}r_{B})^{2\alpha})^{\frac{n}{2\alpha}+2}}dsdu\\
    &\quad+r_{B}^{n}\langle f\rangle_{B}\int_{(2^{j}r_{B})^{2\alpha}}^{\infty}\int_{0}^{r_{B}^{2\alpha}}\frac{1}{(s+u+(2^{j}r_{B})^{2\alpha})^{\frac{n}{2\alpha}+2}}dsdu\\
    &\lesssim r_{B}^{n}\langle f\rangle_{B}\left[\int_{r_{Q}^{2\alpha}}^{(2^{j}r_{B})^{2\alpha}}\frac{r_{B}^{2\alpha}}{(u+(2^{j}r_{B})^{2\alpha})^{\frac{n}{2\alpha}+2}}du+\int_{0}^{r_{B}^{2\alpha}}\frac{1}{(s+(2^{j}r_{B})^{2\alpha})^{\frac{n}{2\alpha}+1}}ds\right]\\
    &\lesssim r_{B}^{n}\langle f\rangle_{B}\left(r_{B}^{2\alpha}(2^{j}r_{B})^{-n-2\alpha}+r_{B}^{2\alpha}(2^{j}r_{B})^{-n-2\alpha}\right)\sim 2^{-j(n+2\alpha)}\langle f\rangle_{B}.\tag{3.17}
\end{align*}
By combining (3.14), (3.16) and (3.17), we obtain the pointwise estimate that for $x\in2^{j+1}B\backslash2^{j}B$,
$$M_{1}(S_{r_{B}^{2\alpha}}(L^{\alpha})f)(x)\lesssim2^{-j(n+2\alpha)}\langle f\rangle_{B}.$$
As previously established, $M_{1}$ is bounded on $L^{2}$. Consequently, this yields the inequality (3.2), which, in combination with (3.13) and Lemma 3.2, implies that $M_{1}$ is of weak type ($1,1$). Combining (3.10), (3.11) and (3.12), we conclude that $M_{V_{a}(e^{-tL^{\alpha}})}$ is of weak type (1,1). As noted earlier, it is known that $V_{a}(e^{-tL^{\alpha}})$ is of weak type $(1,1)$ in the non-weighted case. Hence, by using Proposition 2.6, we estimate that
\begin{align*}
\textnormal{I}_{2}&\lesssim\sum_{j}\mathcal{A}_{\mathcal{S}_{j}}(f\chi_{3Q_{j}})(x)=\sum_{j}\sum_{P\in\mathcal{S}_{j}}\langle f\rangle_{3P}\chi_{P}(x)\\
&=\sum_{j}\sum_{P\in\mathcal{S}_{j}}\langle f\rangle_{3P}\left(1+\frac{r_{p}}{\rho(x_{P})}\right)^{-\tau}\left(1+\frac{r_{p}}{\rho(x_{P})}\right)^{\tau}\chi_{P}(x)\\
&\lesssim \mathcal{A}_{\mathcal{S}_{j}}^{\rho,\tau}f(x)=\mathcal{A}_{\mathcal{S}_{j}}^{\rho,\tau}(f\chi_{3Q_{j}})(x),\tag{3.18}
\end{align*}
where we use that for $P\in\mathcal{S}_{j}\subseteq\mathcal{D}(Q_{j})$, we have $r_{P}\leq r_{Q_{j}}\sim\rho(x_{j})\sim\rho(x_{P})$. Combining (3.18) with (3.9), we complete the proof of Lemma 3.1, which in turn established Theorem 1.1.
\end{proof}
\begin{rem}
    The proof presented here differs from the approach adopted by Wang, Zhao, Li and Liu in \textnormal{\cite{WZLL25}}, where they established one-weight $L^{p}$ boundedness for $V_{a}(e^{-tL^{\alpha}})$. This enabling us to obtain not only $L^{p}$ boundedness, but also quantitative bounds and two-weight results. Our argument relies primarily on the kernel estimate given in Lemma 2.7, whereas their proof in \textnormal{\cite{WZLL25}} requires additional estimate of $|\partial^{m}K_{\alpha,t}^{L}(x+h,y)-\partial^{m}K_{\alpha,t}^{L}(x,y)|$ beyond this.
\end{rem}
\section{Proof of Theorem 1.2}
\quad In this section, we employ the pointwise estimate established in Lemma 3.1 to analyze the mixed weak-type inequalities for $V_{a}(e^{-tL^{\alpha}})$. Our argument relies on an extrapolation theorem and the verification of an $L^{p}$ boundedness inequality via good-$\lambda$ techniques. The following lemmas will be employed in the subsequent proof. First, the mixed weak-type estimate for $M^{\rho,\sigma}$ is established by the following result.
\begin{lem}
\textnormal{(Berra, Pradolini and Quijano \cite{BPQ25})} 
    Let $u\in A_{1}^{\rho}$ and $v\in A_{\infty}^{\rho}$. Then there exist $\sigma\geq0$ and $C>0$, such that the inequality
    $$uv\left(\left\{x\in\mathbb{R}^{n}:\frac{M^{\rho,\sigma}(fv)(x)}{v(x)}>t\right\}\right)\leq\frac{C}{t}\int_{\mathbb{R}^{n}}|f(x)|u(x)v(x)dx$$
holds for any $t>0$.
\end{lem}
In the same work, the authors also prove the following extrapolation theorem.
\begin{lem}
    \textnormal{(Berra, Pradolini and Quijano \cite{BPQ25})} Let $\mathcal{F}$ be a family of pairs of function satisfying that there exist $0<p_{0}<\infty$ such that the inequality 
    $$\int_{\mathbb{R}^{n}}|f(x)|^{p_{0}}w(x)dx\leq C\int_{\mathbb{R}^{n}}|g(x)|^{p_{0}}w(x)dx$$
    holds for any $w\in A_{\infty}^{\rho}$, for any pair $(f,g)\in\mathcal{F}$ such that the left hand side is finite. Then there exists constant $\tilde{C}>0$ that the inequality
    $$\left\|\frac{f}{v}\right\|_{L^{1,\infty}(uv)}\leq \tilde{C}\left\|\frac{g}{v}\right\|_{L^{1,\infty}(uv)}$$
    holds for any $u\in A_{1}^{\rho}$, $v\in A_{\infty}^{\rho}$ and $(f,g)\in\mathcal{F}$.
\end{lem}
\begin{rem}
    Indeed, in the proof of Lemma 4.2, the authors establish that the result holds under the assumption that such $L^{p_{0}}$ bound is satisfied for every $p_{0}$, rather than for any single fixed $p_{0}$. Naturally, we maintain that the conclusion remains valid; however, it requires a result analogous to that of Cruz-Uribe, Martell and P\'{e}rez \textnormal{\cite{CMP04}} for general $A_{\infty}$ weights, which establishes that if the $L^{p_{0}}$ bound holds for some $p_{0}$, then it extends to all $0<p<\infty$. In this paper, we apply Lemma 4.2 only in the case where the $L^{p}$ bound is assumed to hold for all $p$ (see Lemma 4.3 below). Therefore, we do not provide a proof that addresses the more general case here.
\end{rem}
We begin by establishing a lemma that will facilitate the extrapolation argument.
\begin{lem}
    For any $\tau,\gamma\geq0$, $w\in A_{\infty}^{\rho}$, $0<p<\infty$, and under the assumptions of Lemma 3.1, there exist constant $C>0$ such that
    $$\int_{\mathbb{R}^{n}}(\mathcal{A}_{\mathcal{S}_{j}}^{\rho,\tau}f(x))^{p}w(x)dx\leq C\int_{\mathbb{R}^{n}}(M^{\rho,\gamma}f(x))^{p}w(x)dx.$$
\end{lem}
\begin{proof}[Proof of Lemma 4.3]
    The key point in the proof is that for a critical cube $Q_{j}$ and the sparse collection $\mathcal{S}_{j}\subseteq\mathcal{D}(Q_{j})$, $\mathcal{A}_{\mathcal{S}_{j}}^{\rho,\tau}f\sim\mathcal{A}_{\mathcal{S}_{j}}f$. Consequently, we may apply a technique introduced by Cejas, Li, Pérez and Rivera-Ríos \cite{CLPR20} for the proof. Firstly, for convenience of notation, we denote $\mathcal{S}_{j}$ simply by $\mathcal{S}$. For any $Q\in\mathcal{S}$, we have
    $$\psi(3Q)=1+\frac{3r_{Q}}{\rho(x_{Q})}\leq1+\frac{3r_{Q_{j}}}{C\rho(x_{j})}\leq A.$$
    It follows that
    \begin{align*}
        \|\mathcal{A}_{\mathcal{S}}^{\rho,\tau}f\|_{L^{p}(w)}^{p}&\leq\sum_{k\in\mathbb{Z}}2^{(k+1)p}w\left(\left\{x:2^{k}<\mathcal{A}_{\mathcal{S}}^{\rho,\tau}f(x)\leq2^{k+1}\right\}\right)\\
        &\lesssim\sum_{k\in\mathbb{Z}}2^{kp}w\left(\left\{x:\mathcal{A}_{\mathcal{S}}^{\rho,\tau}f(x)>2^{k}\right\}\right)\\
        &\lesssim\sum_{k\in\mathbb{Z}}2^{kp}w\left(\left\{x:\mathcal{A}_{\mathcal{S}}^{\rho,\tau}f(x)>2^{k},M^{\rho,\gamma}f(x)\leq2^{k}\right\}\right)+\sum_{k\in\mathbb{Z}}2^{kp}w\left(\left\{x:M^{\rho,\gamma}f(x)>2^{k}\right\}\right)\\
        &=:\textnormal{I+II}.\tag{4.1}
    \end{align*}
    \underline{Estimate of II}: 
    \begin{align*}
        \|M^{\rho,\gamma}f\|_{L^{p}(w)}^{p}&=\int_{0}^{\infty}w\left(\left\{x:M^{\rho,\gamma}f(x)>\lambda\right\}\right)\lambda^{p-1}d\lambda\\
        &=\sum_{k\in\mathbb{Z}}\int_{2^{k}}^{2^{k+1}}w\left(\left\{x:M^{\rho,\gamma}f(x)>\lambda\right\}\right)\lambda^{p-1}d\lambda\\
        &\gtrsim\sum_{k\in\mathbb{Z}}2^{kp}w\left(\left\{x:M^{\rho,\gamma}f(x)>2^{k+1}\right\}\right).
    \end{align*}
    Thus, we obtain that
    \begin{equation}
        \textnormal{II}\lesssim\|M^{\rho,\gamma}f\|_{L^{p}(w)}^{p}.\tag{4.2}
    \end{equation}
    \underline{Estimate of \text{I}}: Split $\mathcal{S}=\bigcup_{m\in\mathbb{Z}}\mathcal{S}_{m}$, where
    $$\mathcal{S}_{m}:=\left\{Q\in\mathcal{S}:2^{m}<\langle f\rangle_{3Q}\leq2^{m+1}\right\}.$$
If $2^{m}\geq A^{\gamma}2^{k}$, then for any $x\in Q\in\mathcal{S}_{m}$, $M^{\rho,\gamma}f(x)\geq\psi(3Q)^{-\gamma}\langle f\rangle_{3Q}>A^{-\gamma}2^{m}\geq 2^{k}$. Set $m_{0}:=\left\lfloor \text{log}_{2}(A^{\gamma})\right\rfloor+1$, then we have
\begin{align*}
    \textnormal{I}&\leq\sum_{k\in\mathbb{Z}}2^{kp}w\left(\left\{x:\sum_{m\leq k+m_{0}}\mathcal{A}_{\mathcal{S}_{m}}^{\rho,\tau}f(x)>2^{k}\left(1-\frac{1}{\sqrt{2}}\right)\sum_{m\leq k+m_{0}}2^{\frac{m-k-m_{0}}{2}}\right\}\right)\\
    &\leq\sum_{k\in\mathbb{Z}}2^{kp}\sum_{m\leq k+m_{0}}w\left(\left\{x:\mathcal{A}_{\mathcal{S}_{m}}^{\rho,\tau}f(x)>\left(1-\frac{1}{\sqrt{2}}\right)2^{\frac{m+k-m_{0}}{2}}\right\}\right).\tag{4.3}
\end{align*}
Denote $b_{m}=\sum_{Q\in\mathcal{S}_{m}}\chi_{Q}$, then $\mathcal{A}_{\mathcal{S}_{m}}^{\rho,\tau}f\leq2^{m+1}A^{-\tau}b_{m}$. Let $\mathcal{S}_{m}^{*}$ be the collection of maximal dyadic cubes in $\mathcal{S}_{m}$, then we have that
\begin{align*}
    w\left(\left\{x:\mathcal{A}_{\mathcal{S}_{m}}^{\rho,\tau}f(x)>\left(1-\frac{1}{\sqrt{2}}\right)2^{\frac{m+k-m_{0}}{2}}\right\}\right)
    &\leq w\left(\left\{x:b_{m}(x)>\frac{\sqrt{2}-1}{2\sqrt{2}}A^{\tau}2^{\frac{-m+k-m_{0}}{2}}\right\}\right)\\
    &=\sum_{Q\in\mathcal{S}_{m}^{*}}w\left(\left\{x\in Q:b_{m}(x)>\frac{\sqrt{2}-1}{2\sqrt{2}}A^{\tau}2^{\frac{-m+k-m_{0}}{2}}\right\}\right).\tag{4.4}
\end{align*}
For $Q\in \mathcal{S}_{m}^{*}$ and $x\in Q$, it follows that
\begin{align*}
    b_{m}(x)&=\sum_{Q^{\prime}\in\mathcal{S}_{m}}\chi_{Q^{\prime}}(x)=\sum_{Q^{\prime}\in\mathcal{S}_{m}(Q)}\chi_{Q^{\prime}}(x)\\
    &\lesssim\sum_{Q^{\prime}\in\mathcal{S}_{m}(Q)}\left(\frac{1}{|Q^{\prime}|}|E_{Q^{\prime}}|\right)^{2}\chi_{Q^{\prime}}(x)\\
    &\lesssim\sum_{Q^{\prime}\in\mathcal{S}_{m}(Q)}\left(\frac{1}{|Q^{\prime}|}\int_{Q^{\prime}}\chi_{E_{Q^{\prime}}}(x)dx\right)^{2}\chi_{Q^{\prime}}(x)\\
    &\leq\left[\bar{M}_{2}(\{\chi_{Q^{\prime}}\}_{Q^{\prime}\in\mathcal{S}_{m}(Q)})(x)\right]^{2},
\end{align*}
where $\mathcal{S}_{m}(Q):=\{Q^{\prime}\in\mathcal{S}:Q^{\prime}\subseteq Q\}$, and the vector-valued extension of the maximal function introduced by Fefferman and stein \cite{FS71} is defined by $\bar{M}_{r}f(x):=\left(\sum_{j=1}^{\infty}(Mf_{j}(x))^{r}\right)^{{1}/{r}}$ for $f=\left\{f_{j}\right\}_{j=1}^{\infty}$. As established in their work \cite{FS71}, it follows that for any $1<r<\infty$, $\bar{M}_{r}:L^{\infty}(\ell^{r})\rightarrow \text{exp}(L^{r})$. Combining this with $\|\{\chi_{Q^{\prime}}\}_{Q^{\prime}\in\mathcal{S}_{m}(Q)}\|_{L^{\infty}(\ell^{r})}\leq1$, we obtain that
\begin{align*}
    \left|\left\{x\in Q:b_{m}(x)>\frac{\sqrt{2}-1}{2\sqrt{2}}A^{\tau}2^{\frac{-m+k-m_{0}}{2}}\right\}\right|&\leq\left|\left\{x\in Q:\bar{M}_{2}(\{\chi_{Q^{\prime}}\}_{Q^{\prime}\in\mathcal{S}_{m}(Q)})(x)>\frac{\sqrt{2}-1}{2\sqrt{2}}A^{\tau}2^{\frac{-m+k-m_{0}}{2}}\right\}\right|\\
    &\lesssim\text{exp}(-c2^{\frac{-m+k+m_{0}}{2}})|Q|.\tag{4.5}
\end{align*}
Since $w\in A_{\infty}^{
\rho}$, there exist $1\leq p<\infty$ that $w\in A_{p}^{\rho}$. Thus by applying Lemma 2.2, we can see there are constants $0<\delta<1$ and $\eta>0$ that
\begin{equation}
    w\left(\left\{x\in Q:b_{m}(x)>\frac{\sqrt{2}-1}{2\sqrt{2}}A^{\tau}2^{\frac{-m+k-m_{0}}{2}}\right\}\right)\lesssim\text{exp}(-c\delta2^{\frac{-m+k-m_{0}}{2}})w(Q),\tag{4.6}
\end{equation}
where we used that for any $Q\in\mathcal{S}_{m}^{*}\subseteq\mathcal{S}$, $\psi_{\eta}(Q)\lesssim B^{\eta}$. Moreover, for any $x\in Q\in\mathcal{S}_{j}^{*}$, we have $M^{\rho,\gamma}f(x)\geq2^{m}\psi(3Q)^{-\gamma}\geq 2^{m}A^{-\gamma}$. Consequently, (4.3) implies that
\begin{align*}
    \textnormal{I}&\lesssim\sum_{k\in\mathbb{Z}}2^{kp}\sum_{m\leq k+m_{0}}\sum_{Q\in\mathcal{S}_{m}^{*}}\text{exp}(-c\delta2^{\frac{-m+k-m_{0}}{2}})w(Q)\\
    &\lesssim\sum_{k\in\mathbb{Z}}2^{kp}\sum_{m\leq k+m_{0}}\text{exp}(-c\delta2^{\frac{-m+k-m_{0}}
    {2}})w\left(\left\{x:M^{\rho,\gamma}f(x)\geq2^{m}A^{-\gamma}\right\}\right)\\
    &=\sum_{m\in\mathbb{Z}}\sum_{l\geq-m_{0}}2^{(m-m_{0})p}\text{exp}(-c\delta2^{\frac{-l-m_{0}}{2}})w\left(\left\{x:M^{\rho,\gamma}f(x)\geq2^{m}A^{-\gamma}\right\}\right)\\
    &\lesssim\sum_{m\in\mathbb{Z}}2^{mp}w\left(\left\{x:M^{\rho,\gamma}f(x)\geq2^{m}A^{-\gamma}\right\}\right)\lesssim\|M^{\rho,\gamma}f\|_{L^{p}(w)}^{p}.
\end{align*}
The finial inequality follows from an estimate analogous to our earlier bound for II.
\end{proof}
We now proceed to the proof of Theorem 1.2.
\begin{proof}[Proof of Theorem 1.2] It follows by Lemma 4.1, Lemma 4.2 and Lemma 4.3 that for any $\tau\geq0$, $u\in A_{1}^{\rho}$, $v\in A_{\infty}^{\rho}$ and $\gamma\geq\sigma$ (in Lemma 4.1), the mixed weak-type inequality 
   \begin{equation} \left\|\frac{\mathcal{A}_{\mathcal{S}_{j}}^{\rho,\tau}(fv)}{v}\right\|_{L^{1,\infty}(uv)}\lesssim\left\|\frac{M^{\rho,\gamma}(fv)}{v}\right\|_{L^{1,\infty}(uv)}\leq C\|f\|_{L^{1}(uv)}\notag
   \end{equation}
holds under the conditions specified in Lemma 3.1. Since $\mathcal{A}_{\mathcal{S}_{j}}^{\rho,\tau}f$ is supported on $Q_{j}$, and $\sum_{j}\chi_{Q_{j}}(x)\leq D$ for some constant $D>0$, it follows that
$$\left\{x:\sum_{j}^{}\frac{\mathcal{A}_{\mathcal{S}_{j}}(f\chi_{3Q_{j}}v)(x)}{v(x)}>\lambda\right\}\subseteq\bigcup_{j}^{}\left\{x:\frac{\mathcal{A}_{\mathcal{S}_{j}}(f\chi_{3Q_{j}}v)(x)}{v(x)}>\frac{\lambda}{D}\right\}.$$
Thus, by Lemma 3.1 and Lemma 4.1 we conclude that for any $t>0$ and $\gamma\geq\sigma$ (in Lemma 4.1), it holds that
\begin{align*}
    uv\left(\left\{x:\frac{V_{a}(e^{-tL^{\alpha}})(fv)(x)}{v(x)}>t\right\}\right)&\lesssim uv\left(\left\{x:\sum_{j}^{}\frac{\mathcal{A}_{\mathcal{S}_{j}}(f\chi_{3Q_{j}}v)(x)}{v(x)}>\frac{t}{2C}\right\}\right)\\
    &\quad+uv\left(\left\{x:\frac{M^{\rho,\gamma}(fv)(x)}{v(x)}>\frac{t}{2C}\right\}\right)\\
    &\lesssim\sum_{j}^{}uv\left(\left\{x:\frac{\mathcal{A}_{\mathcal{S}_{j}}(f\chi_{3Q_{j}}v)(x)}{v(x)}>\frac{t}{2CD}\right\}\right)+\frac{1}{t}\|f\|_{L^{1}(uv)}\\
    &\lesssim\sum_{j}^{}\frac{1}{t}\|f\chi_{3Q_{j}}\|_{L^{1}(uv)}+\frac{1}{t}\|f\|_{L^{1}(uv)}\lesssim\frac{1}{t}\|f\|_{L^{1}(uv)},
\end{align*}
and thus Theorem 1.2 is proved. In the last inequality, we used $\sum_{j}\chi_{3Q_{j}}(x)\leq C$.
\end{proof}

Our proof of Theorem 1.2 is different from the approach taken by Wen and Wu in \textnormal{\cite{WW25}}. In their work, they established the mixed weak-type inequality for $V_{a}(e^{-tL})$ by bounding the local part of $V_{a}(e^{-tL})$ by $V_{a}(e^{t\Delta})$ and the global part by $M^{\rho,\theta}$. They then handled $V_{a}(e^{t\Delta})$ by employing a new weight class $A_{p}^{\rho,loc}$, which is closely related to the general $A_{p}$ weights. The advantage of our extrapolation-based proof lies in its reliance primarily on the kernel estimate in Lemma 2.7. In contrast to the proof in \textnormal{\cite{WW25}}, our approach does not require additional estimates concerning $|\partial^{m}K_{\alpha,t}^{L}(x+h,y)-\partial^{m}K_{\alpha,t}^{L}(x,y)|$.
\begin{rem}
    We prove that, under the assumptions of Lemma 3.1, the operator $\mathcal{A}_{\mathcal{S}_{j}}^{\rho,\tau}$ is dominated by some maximal operator $M^{\rho,\gamma}$ in weighted $L^{p}$ space. The crucial observation is that for each cube $Q\in\mathcal{S}_{j}$, the quantity $\psi(Q)$ remains approximately constant. This allows us to adapt the method developed in \textnormal{\cite{CLPR20}} for non-critical radius case to $\mathcal{A_{\mathcal{S}}}$ and $M$. However, if this condition is removed, it remains unclear whether the conclusion of Lemma 3.1 still holds. Therefore, for general sparse operator associated with critical radius function $\mathcal{A}_{\mathcal{S}}^{\rho,\tau}$, we have not established the mixed weak-type inequality, nor have we obtained a weak $(1,1)$ inequality.
\end{rem}
\noindent\textbf{Acknowledgement}\\
I would like to express my sincere gratitude to my supervisor, Professor Y. Tsutsui, for his invaluable guidance and support throughout this research. His insights and eccouragement were essential to the completion of this work.\\

\noindent\textbf{Funding}\\
This research did not receive any specific grant from funding agencies in the public, commercial, or not-for-profit sectors.\\

\noindent\textbf{Availability of data and materials}\\
No datasets were generated or analysed during the current study.\\

\noindent\textbf{Competing interests}\\
The authors declare no competing interests.


\begin{thebibliography}{99}
\begin{sloppypar}
\bibitem{A07}{P. Auscher, ``\textit{On necessary and sufficient conditions for $L^{p}$-estimates of Riesz transforms associated to elliptic operators on $\mathbb{R}^{n}$ and related estimates}" Mem. Amer. Math. Soc. \textbf{186} (2007), no. 871, xviii+75 pp.}
\bibitem{B89}{J. Bourgain, ``\textit{Pointwise ergodic theorems for arithmetric sets}" Publ. Math. Inst. Hautes \'{E}tudes
 Sci. \textbf{69} (1989), no. 1, 5–45.}
\bibitem{BBD21}{T. A. Bui, T. Q. Bui and X. T. Duong, ``\textit{Quantitative weighted estimates for some singular integrals related to critical functions}" J. Geom. Anal. \textbf{31} (2021), no. 10, 10215–10245.}
\bibitem{BBD22}{T. A. Bui, T. Q. Bui and X. T. Duong, ``\textit{Quantitative estimates for square functions with new class of weights}" Potential Anal. \textbf{57} (2022), no. 4, 545–569.}
\bibitem{BFHR13}{J. J. Betancor, J. C. Fariña, E. Harboure and L. Rodríguez-Mesa, ``\textit{$L^{p}$-boundedness properties of variation operators in the Schrödinger setting}" Rev. Mat. Complut. \textbf{26} (2013), no. 2, 485–534.}
\bibitem{BHQ20}{B. Bongioanni, E. Harboure and P. Quijano, ``\textit{Two weighted inequalities for operators associated to a critical radius function}" Illinois J. Math. \textbf{64} (2020), no. 2, 227–259.}
\bibitem{BPQ25}{F. Berra, G. Pradolini and P. Quijano, ``\textit{Sawyer estimates of mixed type for operators associated to a critical radius function}" Rev. Mat. Complut. (2025), to appear.}
\bibitem{CDHL18}{Y. Chen, Y. Ding, G. Hong and H. Liu, ``\textit{Weighted jump and variational inequalities for rough operators}" J. Funct. Anal. \textbf{274} (2018), no. 8, 2446–2475.}
\bibitem{CJRW00}{J. Campbell, R. Jones, K. Reinhold and M. Wierdl, ``\textit{Oscillation and variation for the Hilbert transform}" Duke Math. J. \textbf{105} (2000), no. 1, 59–83.}
\bibitem{CLPR20}{M. E. Cejas, K. Li, C. P\'{e}rez and I. P. Rivera-Ríos, ``\textit{Vector-valued operators, optimal weighted estimates and the $C_{p}$ condition}" Sci. China Math. \textbf{63} (2020), no. 7, 1339–1368.}
\bibitem{CMP04}{D. Cruz-Uribe, J. M. Martell and C. P\'{e}rez, ``\textit{Extrapolation from $A_{\infty}$ weights with applications}" J. Funct. Anal. \textbf{213} (2004), no. 2, 412-439.}
\bibitem{CMP05}{D. Cruz-Uribe, J. M. Martell and C. P\'{e}rez, ``\textit{Weighted weak-type inequalities and a conjecture of Sawyer}" Int. Math. Res. Not. \textbf{2005} (2005), no. 30, 1849–1871.}
\bibitem{DZ99}{J. Dziubański and J. Zienkiewicz, ``\textit{Hardy space $H^{1}$ associated to Schrödinger operator with potential satisfying reverse Hölder inequality}" Rev. Mat. Iberoamericana \textbf{15} (1999), no. 2, 279–296.}
\bibitem{FS71}{C. Fefferman and E. M. Stein, ``\textit{Some Maximal Inequalities}" Amer. J. Math. \textbf{93} (1971), 107–115.}
\bibitem{G02}{A. Grigor'yan, ``\textit{Heat kernels and function theory on metric measure spaces}" Contemp. Math. \textbf{338} (2002), 143–172.}
\bibitem{L76}{D. Lépingle, ``\textit{La variation d’ordre p des semi-martingales}" Z. Wahrsch. verw. Gebiete \textbf{36} (1976), 295–316.}
\bibitem{L13}{A. K. Lerner, ``\textit{On an estimate of Calderón-Zygmund operators by dyadic positive operators}" J. Anal. Math. \textbf{121} (2013), 141–161.}
\bibitem{LOP19}{K. Li, S. Ombrosi and C. P\'{e}rez, ``\textit{Proof of an extension of E. Sawyer's conjecture about weighted mixed weak-type estimates}" Math. Ann. \textbf{374} (2019), no. 1-2, 907–929.}
\bibitem{LWQZ22}{P. Li, Z. Wang, T. Qian and C. Zhang, ``\textit{Regularity of fractional heat semigroup associated with Schrödinger operators}" Fractal Fract. \textbf{112} (2022), no. 6, 112-162.}
\bibitem{S82}{B. Simon, ``\textit{Schrödinger semigroups}" Bull. Amer. Math. Soc. (N.S.) \textbf{7} (1982), no. 3, 447–526.}
\bibitem{S85}{E. Sawyer, ``\textit{A weighted weak type inequality for the maximal function}" Proc. Amer. Math. Soc. \textbf{93} (1985), no. 4, 610–614.}
\bibitem{S95}{Z. Shen, ``\textit{$L^{p}$ estimates for Schrödinger operators with certain potentials}" Ann. Inst. Fourier (Grenoble) \textbf{45} (1995), no. 2, 513–546.}
\bibitem{TZ16}{L. Tang and Q. Zhang, ``\textit{Variation operators for semigroups and Riesz transforms acting on weighted $L^{p}$ and BMO spaces in the Schrödinger setting}" Rev. Mat. Complut. \textbf{29} (2016), no. 3, 559–621.}
\bibitem{WW25}{Y. Wen and H. Wu, ``\textit{Weighted variational inequalities for heat semigroups associated with Schrödinger operators related to critical radius functions}" J. Fourier Anal. Appl. \textbf{31} (2025), no. 3, Paper No. 30, 39 pp.}
\bibitem{WZLL25}{Z. Wang, N. Zhao, P. Li and Y. Liu, ``\textit{Weighted boundedness of variation operators for fractional heat semigroups related to the Schrödinger operator}" J. Pseudo-Differ. Oper. Appl. \textbf{16} (2025), no. 2, Paper No. 40, 39 pp.}
\bibitem{ZW15}{J. Zhang and H. Wu, ``\textit{Oscillation and variation inequalities for the commutators of singular integrals with Lipschitz functions}" J. Inequal. Appl. 2015, 2015:214, 21 pp.}
\bibitem{ZW16}{J. Zhang and H. Wu, ``\textit{Oscillation and variation inequalities for singular integrals and commutators on weighted Morrey spaces}" Front. Math. China \textbf{11} (2016), 423–447.}
\end{sloppypar}
\end{thebibliography}
\end{document}